\documentclass{amsart}
\usepackage{amssymb}
\usepackage{mathtools}
\usepackage{url}

\renewcommand{\leq}{\leqslant}
\renewcommand{\geq}{\geqslant}
\newcommand{\ZZ}{\mathbb{Z}}
\newcommand{\RR}{\mathbb{R}}
\newcommand{\divides}{\mathrel{|}}
\newcommand{\ndivides}{\mathrel{\nmid}}

\DeclareMathOperator{\Span}{span}

\DeclareMathOperator{\res}{res}
\newcommand{\norm}[1]{\left \lVert #1 \right \rVert}
\newcommand{\tnorm}[1]{\lVert #1 \rVert}
\newcommand{\supnorm}[1]{\norm{#1}_{\infty}}
\newcommand{\tsupnorm}[1]{\tnorm{#1}_{\infty}}

\newtheorem{thm}{Theorem}[section]
\newtheorem{lem}[thm]{Lemma}
\newtheorem{prop}[thm]{Proposition}

\theoremstyle{definition}

\theoremstyle{remark}
\newtheorem{rem}[thm]{Remark}

\numberwithin{equation}{section}

\begin{document}

\title{A deterministic algorithm for finding $r$-power divisors}

\author[D. Harvey]{David Harvey}
\address{School of Mathematics and Statistics, University of New South Wales, Sydney NSW 2052, Australia}
\email{d.harvey@unsw.edu.au}
\thanks{The first author was supported by the Australian Research Council (grant FT160100219).}

\author[M. Hittmeir]{Markus Hittmeir}
\address{SBA Research, Floragasse 7, A-1040 Vienna}
\email{mhittmeir@sba-research.org}
\thanks{SBA Research (SBA-K1) is a COMET Centre within the framework of COMET – Competence Centers for Excellent Technologies Programme
	and funded by BMK, BMDW, and the federal state of Vienna. The COMET Programme is managed by FFG}

\begin{abstract}
   Building on work of Boneh, Durfee and Howgrave-Graham,
   we present a deterministic algorithm that provably finds all integers $p$
   such that $p^r \divides N$ in time $O(N^{1/4r+\epsilon})$
   for any $\epsilon > 0$.
   For example, the algorithm can be used to test squarefreeness
   of $N$ in time $O(N^{1/8+\epsilon})$;
   previously, the best rigorous bound for this problem was
   $O(N^{1/6+\epsilon})$, achieved via the Pollard--Strassen method.
\end{abstract}

\maketitle

\section{Introduction}

\subsection{Statement of main result}

Let $r$ be a positive integer.
In this paper we study the problem of finding all
$r$\nobreakdash-power divisors of a given positive integer~$N$,
i.e., all positive integers $p$ such that ${p^r \divides N}$.
Throughout the paper we write $\lg x \coloneqq \log_2 x$,
and unless otherwise specified,
the ``running time'' of an algorithm refers to
the number of bit operations it performs,
or more formally, the number of steps executed by a deterministic
multitape Turing machine \cite{Pap-complexity}.
We always assume the use of fast (quasilinear time)
algorithms for basic integer arithmetic, i.e.,
for multiplication, division and GCD
(see for example \cite{MCA_2013} or \cite{BZ-mca}).

Our main result is the following theorem.
\begin{thm}
   \label{thm:main}
   There is an explicit deterministic algorithm with the following properties.
   It takes as input an integer $N \geq 2$
   and a positive integer $r \leq \lg N$.
   Its output is a list of all positive integers $p$ such that $p^r \divides N$.
   Its running time is
   \begin{equation}
      \label{eq:main-bound}
      O\left( N^{1/4r} \cdot
      \frac{(\lg N)^{10+\epsilon}}{r^3} \right).
   \end{equation}
\end{thm}

Note that whenever we write $\epsilon$ in a complexity bound,
we mean that the bound holds for all $\epsilon > 0$,
where the implied big-$O$ constant may depend on $\epsilon$.

The integers $p$ referred to in Theorem \ref{thm:main} need not be prime.
Of course, if $p$ is a composite integer found by the algorithm,
then the algorithm will incidentally determine the
complete factorisation of $p$,
as the prime divisors $\ell$ of $p$ must also satisfy $\ell^r \divides N$.

The hypothesis $r \leq \lg N$ does not really limit
the applicability of the theorem: if $r > \lg N$ then the problem is trivial,
as the only possible $r$-power divisor is $1$.

Theorem \ref{thm:main} is intended primarily as a theoretical result.
For fixed $r$ the complexity is $O(N^{1/4r+\epsilon})$,
which is fully exponential in $\lg N$,
so the algorithm cannot compete asymptotically with subexponential factoring
algorithms such as the elliptic curve method (ECM)
or the number field sieve (NFS).
Furthermore, experiments confirm that for small $r$ the algorithm is
grossly impractical compared to general-purpose factoring routines implemented
in modern computer algebra systems.

\subsection{Previous work}

At the core of our algorithm is a generalisation of
Coppersmith's method \cite{Cop-lowexponent}
introduced by Boneh, Durfee and Howgrave-Graham \cite{BDH-prq}.
We refer to the latter as the \emph{BDHG algorithm}.
Coppersmith's seminal work showed how to use lattice methods to quickly
find all divisors of $N$ in certain surprisingly large intervals.
To completely factor $N$,
one simply applies the method to a sequence of intervals that covers
all possible divisors up to $N^{1/2}$.
Each interval is searched in polynomial time,
so the overall complexity is governed by the number of such intervals,
which turns out to be $O(N^{1/4+\epsilon})$.
The BDHG algorithm adapts Coppersmith's method
to the case of $r$-power divisors.
The relationship between our algorithm and the BDHG algorithm is
discussed in Section \ref{sec:BDHG} below.

We emphasise that, unlike factoring algorithms such as ECM or NFS,
whose favourable running time analyses depend on heuristic assumptions,
the complexity bound in Theorem \ref{thm:main} is \emph{rigorously analysed}
and fully \emph{deterministic}.
Under these restrictions, for $r \geq 2$ it is asymptotically superior
to all previously known complexity bounds
for the problem of finding $r$-power divisors.

Its closest competitors are the algorithms of Strassen \cite{Str-factoring}
and Pollard \cite{Pol-factorization}.
These algorithms can be used to find all divisors of $N$ less than a
given bound $B$ in time $O(B^{1/2+\epsilon})$.
If $p^r \divides N$, say $N = p^r q$,
then either $p \leq N^{1/(r+1)}$ or $q \leq N^{1/(r+1)}$,
so the Pollard--Strassen method can be used to find $p$ or $q$,
and hence both, in time $O(N^{1/2(r+1)+\epsilon})$.
For example, taking $r = 2$, these algorithms can find all \emph{square}
divisors of $N$ in time $O(N^{1/6+\epsilon})$,
whereas our algorithm finds all square divisors in time $O(N^{1/8+\epsilon})$.

There is one special case in which the Pollard--Strassen approach still wins.
If one knows in advance that $p$ is relatively small,
say $p < N^c$ for some $c \in (0, 1/2r)$,
then the Pollard--Strassen method has complexity $O(N^{c/2+\epsilon})$,
which is better than the bound in Theorem \ref{thm:main}.
Our algorithm can also take advantage of the information that $p < N^c$,
but unfortunately this yields only a constant-factor speedup.

Another point of difference is the space complexity.
The space required by the algorithm in Theorem \ref{thm:main} is
only polynomial in $\lg N$ (we will not give the details of this analysis),
whereas for the Pollard--Strassen method the space complexity is
the same as the time complexity, up to logarithmic factors.

In connection with the case $r = 2$,
two other works are worth mentioning.
Booker, Hiary and Keating \cite{BHK-squarefree} describe a
subexponential time algorithm that can sometimes prove that
a given integer $N$ is squarefree,
with little or no knowledge of its factorisation.
This algorithm is not fully rigorous,
as its analysis depends on (among other things)
the Generalised Riemann Hypothesis.
Peralta and Okamoto \cite{PO-ecm} present a speedup of the ECM method
for integers of the form $N = p^2 q$.
Again this result is not fully rigorous,
because it depends on standard conjectures concerning the distribution of
smooth numbers in short intervals,
just as in Lenstra's original ECM algorithm.

The case $r = 1$ corresponds to the ordinary factoring problem,
and in this case our algorithm is essentially
equivalent to Coppersmith's method.
As mentioned above, the complexity is $O(N^{1/4+\epsilon})$,
which does not improve on known results;
currently, the fastest known deterministic factoring method
has complexity $O(N^{1/5+\epsilon})$ \cite{HH-onefifthloglog}.
(It is interesting to ask whether the ideas behind \cite{HH-onefifthloglog}
can be used to improve Theorem \ref{thm:main} when $r \geq 2$.
Our inquiries in this direction have been so far unsuccessful.)

In fact, when $r = 1$, Theorem \ref{thm:main} gives the more precise
complexity bound $O(N^{1/4} (\lg N)^{10+\epsilon})$.
It is apparently well known that Coppersmith's method has complexity
$O(N^{1/4} (\lg N)^C)$ for some constant $C > 0$,
but to the best of our knowledge,
this is the first time in the literature that a particular value of
$C$ has been specified.
On the other hand,
we have not tried particularly hard to optimise the value of $C$,
and it is likely that it can be improved.
(One possible improvement is outlined in Remark \ref{rem:gaussian}.)

\subsection{Relationship to the BDHG algorithm}
\label{sec:BDHG}

The authors of \cite{BDH-prq} were mainly interested in
cryptographic applications,
and this led them to focus on the case that $N = p^r q$
where $p$ and $q$ are roughly the same size.
In this setting, they show that their algorithm is faster than ECM
when $r \approx (\log p)^{1/2}$,
and that it even runs in polynomial time when $r$ is as large as $\log p$.

In this paper we take a different point of view:
our goal is to determine the worst-case complexity,
without any assumptions on the size of $p$, $q$ or $r$.

To illustrate what difference this makes,
consider again the case $r = 2$.
This case is mentioned briefly in Section 6 of \cite{BDH-prq}.
The authors point out that if $N = p^2 q$,
where $p$ and $q$ are known to be about the same size,
i.e., both $p$ and $q$ are within a constant factor of $N^{1/3}$,
then the running time of their method is $O(N^{1/9+\epsilon})$,
i.e., the number of search intervals is $O(N^{1/9+\epsilon})$.
However, in our more general setup, this is \emph{not} the worst case.
Rather, the worst case running time is $O(N^{1/8+\epsilon})$,
which occurs when searching for
$p \sim N^{1/4}$ and $q \sim N^{1/2}$.

More generally,
for $r \geq 1$ the worst case running time of $O(N^{1/4r+\epsilon})$
stated in Theorem \ref{thm:main}
occurs when $p \sim N^{1/2r}$ and $q \sim N^{1/2}$.
By contrast, in the ``balanced'' situation considered in \cite{BDH-prq},
where $p, q \sim N^{1/(r+1)}$,
one can show that the running time is only $O(N^{1/(r+1)^2+\epsilon})$
(see Remark \ref{rem:worst-case}, and take $\theta = r/(r+1)$).

Although the core of our algorithm is essentially the same as
the BDHG algorithm,
our more general perspective requires us to make
a few changes to their presentation.
For instance, we cannot take the lattice dimension to be $d \approx r^2$
(as is done in the main theorem of \cite{BDH-prq}),
because this choice is suboptimal when $r$ is small and fixed.
Additional analysis is required to deal with potentially small values
of $p$ and $q$,
and in general we must take more care than \cite{BDH-prq}
in estimating certain quantities throughout the argument.
For these reasons, we decided to give a self-contained presentation,
not relying on the results in \cite{BDH-prq}.

\subsection{Root-finding}

An important component of our algorithm,
and of all algorithms pursuing Coppersmith's strategy,
is a subroutine for finding all integer roots
of a polynomial with integer coefficients.
This problem has received extensive attention in the literature,
but we were unable to locate a clear statement
of a deterministic complexity bound suitable for our purposes.
For completeness, in Appendix~\ref{sec:root-finding} we give a detailed proof
of the following result.
For a polynomial $f \in \ZZ[x]$,
we write $\supnorm{f}$ for the maximum of the absolute values
of the coefficients of $f$.
\begin{thm}
   \label{thm:root-finding}
   Let $b \geq n \geq 1$ be integers.
   Given as input a polynomial $f \in \ZZ[x]$ of degree $n$
   such that $\supnorm{f} \leq 2^b$,
   we may find all of the integer roots of $f$ in time
   \[
      O(n^{2+\epsilon} b^{1+\epsilon}).
   \]
\end{thm}
Note that this complexity bound is much stronger than what is needed for
the application in this paper.
However, it is still not quasilinear in the size of the input,
which is $O(nb)$.
For further discussion,
see Remarks \ref{rem:quasilinear1} and \ref{rem:quasilinear2}.

\subsection*{Acknowledgments}

The authors would like to thank Joris van der Hoeven for helpful discussions
on the root finding problem.

\section{Searching one interval}

In this section we recall the strategy of \cite{BDH-prq}
for finding all integers $p$ in a prescribed interval
$P - H \leq p \leq P + H$ such that $p^r \divides N$,
provided that $H$ is not too large.
We will prove the following theorem.
\begin{thm}
   \label{thm:one-interval}
   There is an explicit deterministic algorithm with the following properties.
   It takes as input positive integers $N$, $r$, $m$, $d$, $P$ and $H$
   such that
   \begin{equation}
      \label{eq:r-bound}
      r \leq \lg N,
   \end{equation}
   \begin{equation}
      \label{eq:m-bound}
      m \leq d/r,
   \end{equation}
   \begin{equation}
      \label{eq:P-bound}
      H < P \leq N^{1/r},
   \end{equation}
   and
   \begin{equation}
      \label{eq:H-bound}
      H^{(d-1)/2} <
         \frac{1}{d^{1/2} \, 2^{(d-1)/4}} \cdot \frac{(P-H)^{rm}}{N^{rm(m+1)/2d}}.
   \end{equation}
   Its output is a list of all integers $p$ in the interval
   $P - H \leq p \leq P + H$ such that $p^r \divides N$.
   Its running time is
   \[
      O\big( d^{7+\epsilon} (\tfrac 1r \lg N)^{2+\epsilon} \big).
   \]
\end{thm}

A key tool needed in the proof of Theorem \ref{thm:one-interval}
is the LLL algorithm:
\begin{lem}
   \label{lem:LLL}
   Let $d \geq 1$ and $B \geq 2$.
   Given as input linearly independent vectors $v_0, \ldots, v_{d-1} \in \ZZ^d$
   such that $\norm{v_i} \leq B$,
   in time
   \[
      O\big( d^{5+\epsilon} (\lg B)^{2+\epsilon} \big)
   \]
   we may find a nonzero vector $w$
   in the lattice $L \coloneqq \Span_\ZZ(v_0, \ldots, v_{d-1})$ such that
   \[
      \norm{w} \leq 2^{(d-1)/4} (\det L)^{1/d}.
   \]
   (Here $\norm{\,\cdot\,}$ denotes the standard Euclidean norm on $\RR^d$.)
\end{lem}
\begin{proof}
   We take $w$ to be the first vector in a reduced basis for $L$
   computed by the LLL algorithm \cite[Prop.~1.26]{LLL-factoring}.
   For the bound on $\norm{w}$, see \cite[Prop.~1.6]{LLL-factoring}.
   (For more recent developments on lattice reduction,
   see for example \cite[Ch.~17]{Gal-cryptomath}.)
\end{proof}

Let $\ZZ[x]_d$ denote the space of polynomials in $\ZZ[x]$
of degree less than $d$.
The first step in the proof of Theorem \ref{thm:one-interval}
is the following proposition,
which uses the LLL algorithm to construct a nonzero polynomial $h \in \ZZ[x]_d$
with relatively small coefficients in a carefully chosen lattice.
\begin{prop}
   \label{prop:one-interval}
   Let $N$, $r$, $m$, $d$, $P$ and $H$ be positive integers
   satisfying \eqref{eq:r-bound}, \eqref{eq:m-bound} and \eqref{eq:P-bound}.
   Define polynomials $f_0, \ldots, f_{d-1} \in \ZZ[x]_d$ by
   \[
      f_i(x) \coloneqq
      \begin{cases}
         N^{m - \lfloor i/r \rfloor} (P + x)^i,   &  0 \leq i < rm, \\
         (P + x)^i,                               & rm \leq i < d.
      \end{cases}
   \]
   Then in time
   \begin{equation}
      \label{eq:find-h-bound}
      O\big( d^{7+\epsilon} (\tfrac1r \lg N)^{2+\epsilon} \big)
   \end{equation}
   we may find a nonzero polynomial
   \[
      h(x) = h_0 + \cdots + h_{d-1} x^{d-1} \in \ZZ[x]_d
   \]
   in the $\ZZ$-span of $f_0, \ldots, f_{d-1}$ such that
   \begin{equation}
      \label{eq:h-sum-bound}
      |h_0| + |h_1| H + \cdots + |h_{d-1}| H^{d-1} <
         d^{1/2} \, 2^{(d-1)/4} H^{(d-1)/2} N^{rm(m+1)/2d}.
   \end{equation}
\end{prop}
\begin{proof}
   Set $\tilde f_i(y) \coloneqq f_i(Hy) \in \ZZ[y]_d$,
   and let $v_i \in \ZZ^d$ be the vector whose $j$-th entry
   (for $j = 0, \ldots, d-1$)
   is the coefficient of $y^j$ in $\tilde f_i(y)$.
   We will apply Lemma \ref{lem:LLL} to the vectors $v_0, \ldots, v_{d-1}$.

   Let
   \[
      B \coloneqq d^{1/2} \, 2^d N^{d/r+1}.
   \]
   We claim that $\norm{v_i} \leq B$ for all $i$.
   First consider the case $0 \leq i < rm$.
   For any $j = 0, \ldots, i$, the coefficient of $y^j$ in
   $\tilde f_i(y) = N^{m-\lfloor i/r \rfloor} (P + Hy)^i$
   is equal to
   \[
      N^{m-\lfloor i/r \rfloor} \tbinom{i}{j} P^{i-j} H^j
         \leq N^{m - i/r + 1} 2^i P^i
         \leq 2^i N^{m+1}
         \leq 2^d N^{d/r+1},
   \]
   where we have used the hypotheses \eqref{eq:P-bound} and \eqref{eq:m-bound}.
   For the case $rm \leq i < d$,
   a similar argument shows that every coefficient
   of $\tilde f_i(y) = (P + Hy)^i$ is bounded above by $2^d N^{d/r}$.
   Therefore every $v_i$ has coordinates bounded by $2^d N^{d/r+1}$,
   and we conclude that $\norm{v_i} \leq B$ for all $i$.

   Next we calculate the determinant of the lattice
   $L \coloneqq \Span_\ZZ(v_0, \ldots, v_{d-1})$,
   or equivalently, the determinant of the $d \times d$ integer matrix
   whose rows are given by the $v_i$.
   Since $\deg \tilde f_i(y) = i$, this is a lower triangular matrix
   whose diagonal entries are given by the leading coefficients
   of the $\tilde f_i(y)$, namely
   \[
      \begin{cases}
         N^{m - \lfloor i/r \rfloor} H^i,   &  0 \leq i < rm, \\
         H^i,                               & rm \leq i < d.
      \end{cases}
   \]
   The determinant is the product of these leading coefficients, i.e.,
   \begin{align*}
      \det L & = H^{1 + 2 + \cdots + (d-1)}
            \underbrace{(N^m \cdots N^m)}_{\text{$r$ terms}}
            \underbrace{(N^{m-1} \cdots N^{m-1})}_{\text{$r$ terms}} \cdots
            \underbrace{(N \cdots N)}_{\text{$r$ terms}} \\
             & = H^{1 + 2 + \cdots + (d-1)} (N^{1 + 2 + \cdots + m})^r \\
             & = H^{d(d-1)/2} N^{rm(m+1)/2}.
   \end{align*}

   Invoking Lemma \ref{lem:LLL}, we may compute a nonzero vector $w \in L$
   such that
   \[
      \norm{w} \leq 2^{(d-1)/4} H^{(d-1)/2} N^{rm(m+1)/2d}
   \]
   in time $O(d^{5+\epsilon} (\lg B)^{2+\epsilon})$.
   Note that this time bound certainly dominates the cost of computing
   the vectors $v_i$ themselves,
   as the $\tilde f_i(y)$ may be computed by starting with
   $\tilde f_0(y) = N^m$,
   and then successively multiplying by $P + Hy$
   and occasionally dividing by $N$.
   The hypotheses \eqref{eq:r-bound} and \eqref{eq:m-bound} imply that
   \[
      \lg B \ll d + (\tfrac{d}{r} + 1) \lg N
      = \big(\tfrac{r}{\lg N} + 1 + \tfrac{r}{d} \big) \tfrac{d}{r} \lg N
      \leq \big(2 + \tfrac1m \big) \tfrac{d}{r} \lg N \ll \tfrac{d}{r} \lg N,
   \]
   so the cost estimate $O(d^{5+\epsilon} (\lg B)^{2+\epsilon})$
   simplifies to \eqref{eq:find-h-bound}.

   The vector $w$ corresponds to a nonzero polynomial
   $\tilde h(y) = \tilde h_0 + \cdots + \tilde h_{d-1} y^{d-1}$
   in the $\ZZ$-span of the $\tilde f_i(y)$.
   Applying the Cauchy--Schwartz inequality to the vectors
   $w = (\tilde h_0, \ldots, \tilde h_{d-1})$ and $(1, \ldots, 1)$ yields
   \[
      |\tilde h_0| + \cdots + |\tilde h_{d-1}|
         \leq d^{1/2} \norm{w} < d^{1/2} \, 2^{(d-1)/4} H^{(d-1)/2} N^{rm(m+1)/2d}.
   \]
   Moreover, each $\tilde h_j$ is divisible by $H^j$,
   so we obtain in turn a polynomial
   $h(x) \coloneqq \tilde h(x/H) \in \ZZ[x]_d$
   in the $\ZZ$-span of the $f_i(x)$.
   Since $h(x) = h_0 + \cdots + h_{d-1} x^{d-1}$
   with $h_j = \tilde h_j / H^j$ for each $j$,
   the estimate \eqref{eq:h-sum-bound} follows immediately.
\end{proof}


Next we show that any $r$-power divisor that is sufficiently close to $P$
corresponds to a root of $h(x)$.
\begin{prop}
   \label{prop:divisor-root}
   Let $N$, $r$, $m$, $d$, $P$ and $H$ be positive integers
   satisfying \eqref{eq:r-bound}, \eqref{eq:m-bound} and \eqref{eq:P-bound},
   and let $h(x) \in \ZZ[x]_d$ be as in Proposition \ref{prop:one-interval}.
   Suppose additionally that \eqref{eq:H-bound} holds,
   and that $p$ is an integer in the interval
   $P - H \leq p \leq P + H$ such that $p^r \divides N$.
   Then $x_0 \coloneqq p - P$ is a root of $h(x)$.
\end{prop}
\begin{proof}
   We claim that $h(x_0)$ is divisible by $p^{rm}$.
   Since $h(x)$ is a $\ZZ$-linear combination of the $f_i(x)$
   (where $f_i(x)$ is defined as in Proposition \ref{prop:one-interval}),
   it is enough to prove that $p^{rm} \divides f_i(x_0)$ for all $i$.
   For the case $0 \leq i < rm$, we have
   $f_i(x_0) = N^{m - \lfloor i/r \rfloor} p^i$.
   Since $p^r \divides N$,
   we have $p^{r (m - \lfloor i/r \rfloor)} p^i \divides f_i(x_0)$,
   and this implies that $p^{rm} \divides f_i(x_0)$
   because $r \lfloor i/r \rfloor \leq i$.
   For the case $i \geq rm$ we have simply $f_i(x_0) = p^i$,
   which is certainly divisible by $p^{rm}$.

   On the other hand, the assumption $-H \leq x_0 \leq H$ together with
   \eqref{eq:h-sum-bound} and \eqref{eq:H-bound} implies that
   \[
      |h(x_0)| \leq |h_0| + \cdots + |h_{d-1}| H^{d-1}
      < (P-H)^{rm} \leq p^{rm}.
   \]
   Since $h(x_0)$ is divisible by $p^{rm}$, this forces $h(x_0) = 0$.
\end{proof}

We may now complete the proof of the main theorem of this section.
\begin{proof}[Proof of Theorem \ref{thm:one-interval}]
   We first invoke Proposition \ref{prop:one-interval},
   with inputs $N$, $r$, $m$, $d$, $P$ and $H$,
   to find a polynomial $h(x)$ satisfying \eqref{eq:h-sum-bound}.
   According to Proposition \ref{prop:divisor-root},
   we may then construct a list of candidates for $p$ by
   finding all integer roots of $h(x)$,
   which we do via Theorem \ref{thm:root-finding}.

   To estimate the complexity of the root-finding step,
   recall from the proof of Proposition \ref{prop:divisor-root}
   that $|h_0| + \cdots + |h_{d-1}| H^{d-1} < (P-H)^{rm}$,
   so certainly $|h_j| < (P-H)^{rm}$ for all $j$, and we obtain
   \[
      \supnorm{h} < (P-H)^{rm} < P^{rm} \leq N^m \leq N^{d/r}.
   \]
   Therefore in Theorem \ref{thm:root-finding} we may take $n \coloneqq d$ and
   $b \coloneqq \lceil \lg(N^{d/r}) \rceil = \lceil \tfrac{d}{r} \lg N \rceil$.
   Note that the hypothesis $b \geq n$ is satisfied due to \eqref{eq:r-bound}.
   The root-finding complexity is thus
   \[
      O(d^{2+\epsilon} (\tfrac dr \lg N)^{1+\epsilon})
         = O(d^{3+\epsilon} (\tfrac 1r \lg N)^{1+\epsilon}),
   \]
   which is negligble compared to the main bound \eqref{eq:find-h-bound}.
   Finally, we must check each candidate for $p$
   to ensure that $p^r \divides N$,
   which again requires negligble time.
\end{proof}

\section{Proof of the main theorem}
\label{sec:main-proof}

We now consider the problem of searching for all integers $p$
such that $p^r \divides N$ in an interval, say $T \leq p \leq T'$,
that is too large to be handled by a single application of
Theorem \ref{thm:one-interval}.
Given $N$, $r$, $T$ and $T'$,
our strategy will be to choose parameters $d$, $m$ and $H$,
and then apply Theorem \ref{thm:one-interval} to a sequence of
subintervals of the form $P - H \leq p \leq P + H$ that cover the
target interval $T \leq p \leq T'$.
The overall running time will depend mainly on the number of subintervals,
so our goal is to make $H$ as large as possible.
On the other hand, to ensure that the hypothesis \eqref{eq:H-bound}
of Theorem \ref{thm:one-interval} is satisfied,
we also require that $H < \tilde H$ where
\begin{equation}
   \label{eq:Htilde}
   \tilde H \coloneqq \frac{1}{d^{1/(d-1)} 2^{1/2}} \cdot
      \frac{T^{2rm/(d-1)}}{N^{rm(m+1)/d(d-1)}} > 0.
\end{equation}
The key issue is therefore to choose $d$ and $m$ to maximise $\tilde H$.
For large $d$ and $m$,
the magnitude of $\tilde H$ depends more or less on the ratio $m/d$;
in fact, one finds that $\tilde H$ is maximised when
$m/d \approx \lg T / \lg N$.
The following result gives a simple formula for $m$ (as a function of $d$)
that is close to the optimal choice,
and a corresponding explicit lower bound for $\tilde H$.
\begin{lem}
   \label{lem:Htilde}
   Let $N$, $r$, $d$ and $T$ be positive integers
   such that $d \geq 2$ and $T \leq N^{1/r}$.
   Let
   \begin{equation}
      \label{eq:m}
      m \coloneqq \left\lfloor \frac{(d-1) \lg T}{\lg N} \right\rfloor,
   \end{equation}
   and let $\tilde H$ be defined as in \eqref{eq:Htilde}.
   Then
   \[
      \tilde H > \tfrac13 N^{\theta^2/r \, - \, 1/(d-1)},
   \]
   where
   \begin{equation}
      \label{eq:theta}
      \theta \coloneqq \frac{r \lg T}{\lg N} \in [0,1]
      \qquad \text{(so that $T = N^{\theta/r}$)}.
   \end{equation}
\end{lem}
\begin{proof}
   The definition of $m$ implies that
   $(d-1)\tfrac{\theta}{r} - 1 < m \leq (d-1) \tfrac{\theta}{r}$,
   so we may write
   \[
      \frac{m}{d-1} = \frac{\theta}{r} - \delta \qquad
      \text{for some } \delta \in [0, \tfrac{1}{d-1}).
   \]
   It is easy to check that
   $d^{1/(d-1)} 2^{1/2} < 3$ for all $d \geq 2$,
   so we find that
   \[
      \tilde H > \tfrac13 N^{2\theta m/(d-1) - rm(m+1)/d(d-1)}.
   \]
   Continuing to estimate the exponent in this inequality, we obtain
   \begin{align*}
      \frac{2\theta m}{d-1} - \frac{rm(m+1)}{d(d-1)}
         & > \frac{2\theta m}{d-1} - \frac{rm(m+1)}{(d-1)^2} \\
         & = \frac{2\theta m}{d-1} - \frac{rm^2}{(d-1)^2} - \frac{rm}{(d-1)^2} \\
         & = 2\theta(\tfrac{\theta}{r} - \delta) - r(\tfrac{\theta}{r} - \delta)^2 - \frac{r(\tfrac{\theta}{r} - \delta)}{d-1} \\
         & = \frac{\theta^2}{r} + r\delta \left(\frac{1}{d-1} - \delta \right) - \frac{\theta}{d-1} \\
         & \geq \frac{\theta^2}{r} - \frac{1}{d - 1},
   \end{align*}
   where the last line follows from the inequalities
   $0 \leq \delta < \tfrac{1}{d-1}$ and $0 \leq \theta \leq 1$.
\end{proof}

We may now estimate the time required to search a
given interval $T \leq p \leq T'$.
\begin{prop}
   \label{prop:big-interval}
   There is an explicit deterministic algorithm with the following properties.
   It takes as input positive integers $N$, $r$, $T$ and $T'$ such that
   \eqref{eq:r-bound} holds (i.e., $r \leq \lg N$) and such that
   \begin{equation}
      \label{eq:T-inequality}
      4^{\sqrt{(\lg N)/r}} \leq T < T' \leq N^{1/r}.
   \end{equation}
   Its output is a list of all integers $p$ in the interval
   $T \leq p \leq T'$ such that $p^r \divides N$.
   Its running time is
   \[
      O\left( \left(\frac{T'-T}{T} \cdot N^{\theta(1-\theta)/r} + 1\right)
               \frac{(\lg N)^{9+\epsilon}}{r^2} \right),
   \]
   where $\theta$ is defined as in \eqref{eq:theta}.
\end{prop}
\begin{proof}
   Set
   \[
      d \coloneqq \lceil \lg N \rceil + 1
   \]
   and define $m$ as in \eqref{eq:m}.
   Equivalently, $m$ is the largest integer such that $N^m \leq T^{d-1}$.
   Note that $d \geq 2$ (since $N \geq 2^r \geq 2$)
   and $m \geq \lfloor \lg T \rfloor \geq 2$
   (since $T \geq 4^{\sqrt 1} = 4$).
   Since $\lg(T^{d-1}) \leq d \lg T \ll (\lg N)^2$,
   we may clearly compute $d$ and $m$ in time $O((\lg N)^{2+\epsilon})$.
   Also, the assumption $T \leq N^{1/r}$ implies that
   $m \leq (d-1)/r \leq d/r$, so \eqref{eq:m-bound} holds.
   
   Let $\tilde H$ be defined as in \eqref{eq:Htilde}.
   Since $d \geq \lg N + 1$, Lemma \ref{lem:Htilde} implies that
   \[
      \tilde H > \tfrac13 N^{\theta^2/r} N^{-1/\lg N} = \tfrac16 N^{\theta^2/r}.
   \]
   Moreover, \eqref{eq:T-inequality} implies that
   $\theta \geq 2 \sqrt{r/\lg N}$,
   so we have $N^{\theta^2/r} \geq N^{4/\lg N} = 16$
   and hence $\tilde H > 16/6 > 2$.

   Let $H$ be the largest integer less than $\tilde H$,
   i.e., $H \coloneqq \big\lceil \tilde H \big\rceil - 1$.
   Then $2 \leq H < \tilde H$, and moreover,
   since $\tilde H > 2$, we also have
   \[
      H \geq \tilde H / 2 > \tfrac{1}{12} N^{\theta^2/r}.
   \]
   We may compute $H$ by first approximating
   the $d(d-1)$-th root of the rational number
   \[
      \tilde H^{d(d-1)} = \frac{T^{2drm}}{d^d 2^{d(d-1)/2} N^{rm(m+1)}},
   \]
   and then taking $d(d-1)$-th powers of nearby integers
   to find the correct value.
   The numerator has bit size at most
   $O(drm \lg T) = O(d^2 \lg N) = O(\lg^3 N)$,
   and the denominator also has bit size at most
   \[
      O(d \lg d + d^2 + rm^2 \lg N)
      = O(d^2 + d m \lg N) = O(d^2 \lg N) = O(\lg^3 N),
   \]
   so this can all be done in time $O((\lg N)^{3+\epsilon})$.

   We now apply Theorem \ref{thm:one-interval} with the parameters
   $N$, $r$, $d$, $m$, $H$, and with successively
   $P = T + H$, $P = T + 3H$, and so on,
   stopping when the interval $[T,T']$ has been exhausted
   by the subintervals $[P-H, P+H]$.
   The hypotheses \eqref{eq:r-bound}, \eqref{eq:m-bound} and \eqref{eq:P-bound}
   have already been checked above,
   and \eqref{eq:H-bound} follows from our choice of $H < \tilde H$
   because $P-H \geq T$.
   The number of subintervals is at most
   \[
      \left\lceil \frac{T' - T}{2H} \right\rceil
         \leq \frac{T' - T}{\frac16 N^{\theta^2/r}} + 1
         = \frac{6(T' - T)}{T} \cdot N^{\theta(1-\theta)/r} + 1.
   \]
   Finally, since $d \ll \lg N$,
   the cost of each invocation of Theorem \ref{thm:one-interval} is
   \[
      O\big( d^{7+\epsilon} (\tfrac 1r \lg N)^{2+\epsilon} \big)
         = O\left( \frac{(\lg N)^{9+\epsilon}}{r^2} \right).
      \qedhere
   \]
\end{proof}
\begin{rem}
   A slightly better choice for $d$ is to take $d \approx \theta \lg N$,
   but this complicates the analysis and only improves the main result
   by a constant factor.
\end{rem}

Finally we may prove the main theorem.
Recall that we are given as input
positive integers $N \geq 2$ and $r \leq \lg N$,
and we wish to find \emph{all} positive integers $p$ such that $p^r \divides N$.
Such divisors $p$ must clearly lie in $[1,N^{1/r}]$.
\begin{proof}[Proof of Theorem \ref{thm:main}]
   Let
   \[
      k \coloneqq \left\lceil 2 \sqrt{\lceil \lg N \rceil/r} \right\rceil.
   \]
   We first check all $p = 2, 3, \ldots, 2^k$ by brute force, i.e.,
   testing directly whether $p^r \divides N$.
   Note that $k$ may certainly be computed in time $O((\lg N)^{1+\epsilon})$.
   To estimate the cost of checking up to $2^k$, observe that
   \[
      k \leq 2 \sqrt{(\lg N)/r + 1} + 1.
   \]
   Let $C > 0$ be an absolute constant such that
   $2\sqrt{x+1} + 1 \leq x/4 + C$ for all $x \geq 1$;
   it follows that $k \leq (\lg N)/4r + C$,
   and hence that $2^k \ll N^{1/4r}$.
   The cost of checking up to $2^k$ is therefore
   $O(N^{1/4r} (\lg N)^{1+\epsilon})$,
   which is negligible compared to \eqref{eq:main-bound}.

   We now apply Proposition \ref{prop:big-interval} to the intervals
   $[2^k, 2^{k+1}]$, $[2^{k+1}, 2^{k+2}]$,
   and so on until we reach $N^{1/r}$,
   taking the last interval to be $[2^j, \lfloor N^{1/r} \rfloor]$
   for suitable $j$.
   Since $k \geq 2\sqrt{(\lg N)/r}$, the precondition \eqref{eq:T-inequality}
   is satisfied.
   For each interval we have $(T'-T)/T = O(1)$,
   and since $\theta \in [0,1]$ we have
   \[
      \theta(1-\theta) \leq \frac14.
   \]
   Therefore the cost of searching each interval is
   \[
      O\left( (N^{1/4r} + 1) \cdot \frac{(\lg N)^{9+\epsilon}}{r^2} \right)
      = O\left( N^{1/4r} \cdot \frac{(\lg N)^{9+\epsilon}}{r^2} \right).
   \]
   Finally, the number of intervals is at most
   $\lceil \lg(N^{1/r}) \rceil = O(\tfrac 1r \lg N)$.
\end{proof}
\begin{rem}
   The use of dyadic intervals in the above proof was only for convenience;
   the same argument would work with intervals $[B^j, B^{j+1}]$
   for any fixed $B > 1$.
\end{rem}
\begin{rem}
   \label{rem:worst-case}
   The expression $N^{\theta(1-\theta)/r}$ achieves its
   maximum value $N^{1/4r}$ at the point $\theta = 1/2$.
   This justifies the claim made in the introduction that the factors $p^r$
   that are ``hardest'' to find are those for which $p \sim N^{1/2r}$.
\end{rem}
\begin{rem}
   \label{rem:gaussian}
   A more careful analysis,
   taking into account the fact that $N^{\theta(1-\theta)/r}$
   is much smaller than $N^{1/4r}$ for most values of $\theta \in [0,1]$,
   shows that the bound \eqref{eq:main-bound} can be improved by a factor
   of $O((\tfrac1r \lg N)^{1/2})$.
   Let us briefly explain this calculation.
   The main contribution to the cost estimate in the above proof is the
   number of subintervals,
   i.e., the sum of the values of $N^{\theta(1-\theta)/r}$
   over the various dyadic intervals.
   It can be shown that this sum
   is essentially a Riemann sum approximating the integral
   \[
      \frac{\log N}{r} \int_0^1 N^{\theta(1-\theta)/r} d\theta.
   \]
   The argument in the proof of Theorem \ref{thm:main}
   amounted to estimating this integral via the trivial bound
   $\int_0^1 N^{\theta(1-\theta)/r} d\theta \leq
   \int_0^1 N^{1/4r} d\theta = N^{1/4r}$.
   A better estimate is obtained by recognising the integrand as a
   truncated Gaussian function, i.e.,
   \begin{align*}
      \int_0^1 N^{\theta(1-\theta)/r} d\theta
      & = \int_{-1/2}^{1/2} N^{(1/4-\alpha^2)/r} d\alpha \\
      & \leq N^{1/4r} \int_{-\infty}^\infty N^{-\alpha^2/r} d\alpha
      = \left(\frac{\pi r}{\log N}\right)^{1/2} N^{1/4r}.
   \end{align*}
\end{rem}

\appendix
\section{Deterministic root finding}
\label{sec:root-finding}

In this section we prove Theorem \ref{thm:root-finding}.
Our root-finding procedure consists of two parts.
In the first part,
we discuss how to deterministically find all integer roots
of a \emph{squarefree} polynomial $f \in \ZZ[x]$.
We mainly follow the approach of Loos \cite{Loos-zeroes},
but we obtain better complexity bounds by employing
faster algorithms for the underlying arithmetic.
In the second part,
we explain how to reduce the general case to the squarefree case.
The reduction depends on computing GCDs in $\ZZ[x]$;
for this purpose we present a rigorous, deterministic variant of
the ``heuristic GCD'' algorithm of Char, Geddes and Gonnet \cite{CGG-gcdheu}.

\subsection{Some preliminary estimates}

For $f, g \in \ZZ[x]$,
let $\res(f,g) \in \ZZ$ denote the resultant of $f$ and $g$.
\begin{lem}
   \label{lem:resultant-bound}
   Let $f, g \in \ZZ[x]$ be nonzero polynomials,
   and let $n \coloneqq \deg f$, $m \coloneqq \deg g$.
   Then
   \[
      \lvert \res(f,g) \rvert \leq (n+1)^{m/2} (m+1)^{n/2} \supnorm{f}^m \supnorm{g}^n.
   \]
\end{lem}
\begin{proof}
   See \cite[Thm.~6.23]{MCA_2013}.
   (The proof uses Hadamard's bound to estimate the determinant
   of the Sylvester matrix associated to $f$ and $g$.)
\end{proof}

\begin{lem}[Mignotte's factor bound]
   \label{lem:mignotte}
   Let $f, g \in \ZZ[x]$ be nonzero polynomials,
   and let $n \coloneqq \deg f$, $m \coloneqq \deg g$.
   If $g$ divides $f$ in $\ZZ[x]$ then
   \[
      \supnorm{g} \leq (n+1)^{1/2} \, 2^m \supnorm{f}.
   \]
\end{lem}
\begin{proof}
   See \cite[Cor.~6.33(ii)]{MCA_2013}.
   (The proof relies on Landau's inequality for the Mahler measure
   of a polynomial.)
\end{proof}

\begin{lem}
   \label{lem:theta-bound}
   For all $X \geq 2$ we have
   \[
      \sum_{p \leq X} \lg p > X/3
   \]
   (where the sum is taken over primes).
\end{lem}
\begin{proof}
   Let $\vartheta(X) \coloneqq \sum_{p \leq X} \log p$ denote the usual
   Chebyshev weighted prime counting function.
   The claim is that $\vartheta(X)/X > \frac13 \log 2$ ($\approx 0.231$)
   for all $X \geq 2$.
   For $X \geq 101$ this follows from \cite[Thm.~10]{RS-primes},
   which states that $\vartheta(X)/X > 0.84$ for $X \geq 101$.
   For $2 \leq X < 101$ the claim may be checked directly,
   for example by inspecting the graph of $\vartheta(X)/X$
   in the reader's favourite computer algebra system.
\end{proof}

\subsection{The squarefree case}

The core idea of Loos' algorithm is the following well-known
$p$-adic Hensel lifting strategy.
\begin{prop}
   \label{prop:hensel}
   Let $p$ be a prime, let $k$ be a positive integer,
   and let $f \in (\ZZ/p^k \ZZ)[x]$ be a polynomial of degree $n \geq 1$.
   Let $u \in \{0, \ldots, p-1\}$,
   and suppose that
   \[
      f(u) \equiv 0 \pmod p, \qquad f'(u) \not\equiv 0 \pmod p.
   \]
   Then there exists a unique
   $v \in \{0, \ldots, p^k - 1\}$ such that
   \[
      v \equiv u \pmod p, \qquad f(v) \equiv 0 \pmod{p^k}.
   \]
   Given $f$ and $u$ as input, we may compute $v$ in time
   \[
      O(n \lg(p^k)^{1+\epsilon}).
   \]
\end{prop}
\begin{proof}
   We argue by induction on $k$.
   If $k = 1$, we simply take $v = u$.
   Now assume that $k \geq 2$ and set $\ell \coloneqq \lceil k/2 \rceil < k$.
   By induction there exists a unique $w \in \{0, \ldots, p^\ell - 1\}$
   such that $w \equiv u \pmod p$ and $f(w) \equiv 0 \pmod{p^\ell}$.
   
   We first establish uniqueness of $v$.
   Suppose that $v$ has the desired properties, i.e.,
   $v \equiv u \pmod p$ and $f(v) \equiv 0 \pmod{p^k}$.
   By the uniqueness of $w$, we must have $v \equiv w \pmod{p^\ell}$,
   say $v = w + p^\ell t$ for some $t \in \{0, \ldots, p^{k-\ell}-1\}$.
   Expanding $f$ around $w$, we find that
   \[
      f(w + x) = f(w) + x f'(w) + x^2 g(x)
   \]
   for some $g \in (\ZZ/p^k \ZZ)[x]$.
   Substituting $x = p^\ell t$,
   and using the fact that $p^{2\ell} \equiv 0 \pmod{p^k}$,
   we deduce that $0 \equiv f(w) + p^\ell t f'(w) \pmod{p^k}$.
   Since $f'(w) \equiv f'(u) \not\equiv 0 \pmod p$,
   we may solve for $t$ to obtain
   \[
      t \equiv \frac{-f(w)/p^\ell}{f'(w)} \pmod{p^{k-\ell}}.
   \]
   This establishes uniqueness of $t \pmod{p^{k-\ell}}$,
   and hence of $v \pmod{p^k}$.
   Moreover, the same calculation gives an explicit formula for $v$,
   proving existence.
   
   To prove the complexity bound,
   suppose that we have already computed $w$ and that we wish to lift to $v$.
   We first apply Horner's rule to compute $f(w)$ and $f'(w)$
   using $O(n)$ arithmetic operations in $\ZZ/p^k\ZZ$.
   Each such operation requires time $O(\lg(p^k)^{1+\epsilon})$.
   Similarly, we may invert $f'(w)$, and hence compute $t$ and $v$,
   in time $O(\lg(p^k)^{1+\epsilon})$.
   Therefore, the time required to deduce $v$ from $w$ is
   $O(n \lg(p^k)^{1+\epsilon})$.
   The contributions from subsequent recursion levels form a geometric series,
   so the total cost of computing $v$ from $u$
   is also $O(n \lg(p^k)^{1+\epsilon})$.
\end{proof}

The next result shows how to find a reasonably small prime $p$
for which the $p$-adic lifting strategy is guaranteed to succeed.
\begin{prop}
   \label{prop:find-p}
   Let $b \geq n \geq 1$ be integers, and let $f \in \ZZ[x]$
   be a squarefree polynomial of degree $n$ such that $\supnorm{f} \leq 2^b$.
   Then in time
   \[
      O(n^{2+\epsilon} b^{1+\epsilon})
   \]
   we may find a prime number
   \[
      p \leq 6nb + 6 n \lg n
   \]
   such that the reduction of $f$ modulo $p$ is nonzero and squarefree
   in $(\ZZ/p\ZZ)[x]$.
\end{prop}
\begin{proof}
   Since $f$ is squarefree, the resultant $D \coloneqq \res(f,f')$ is nonzero.
   Our goal is to find a prime $p$ such that $p \ndivides D$.

   First, by Lemma \ref{lem:resultant-bound} we have
   \[
      \lvert D \rvert \leq (n+1)^{(n-1)/2} n^{n/2} \supnorm{f}^{n-1} \supnorm{f'}^n.
   \]
   One easily checks that $(n+1)^{n-1} \leq n^n$ for all $n \geq 1$.
   Since $\supnorm{f'} \leq n \supnorm{f}$, we obtain
   \begin{equation}
      \label{eq:D-bound}
      |D| \leq n^{2n} \, 2^{2nb}.
   \end{equation}

   On the other hand, let $X \coloneqq 6bn + 6n \lg n$.
   If $D$ is divisible by all primes $p \leq X$,
   then $D$ is divisible by their product,
   so
   \[
      \lg |D| \geq \sum_{p \leq X} \lg p > X/3 = 2nb + 2n \lg n
   \]
   by Lemma \ref{lem:theta-bound}.
   This contradicts \eqref{eq:D-bound},
   so we conclude that there must exist a prime $p \leq X$
   such that $p \ndivides D$.
   To actually find such a prime, we run the following algorithm.

   \textit{Step 1 (list primes).}
   Make a list of all primes $p \leq Y$ for
   $Y \coloneqq 6nb + 6n \lceil \lg n \rceil = O(nb)$.
   Using the sieve of Eratosthenes,
   this requires time $O(Y^{1+\epsilon}) = O(n^{1+\epsilon} b^{1+\epsilon})$.

   \textit{Step 2 (reduce $f$ modulo primes).}
   Let $f_p \in (\ZZ/p\ZZ)[x]$ denote the reduction of $f$ modulo $p$.
   We compute $f_p$ for all $p \leq Y$
   by applying a fast simultaneous modular reduction algorithm
   \cite[Thm.~10.24]{MCA_2013} (i.e., using a remainder tree)
   to each coefficient of $f$.
   The bit size of the product of the primes is
   $O(\vartheta(Y)) = O(Y) = O(nb)$,
   and the number of primes is certainly $O(nb)$,
   so the cost per coefficient is $O(n^{1+\epsilon} b^{1+\epsilon})$.
   The total cost over all coefficients is therefore
   $O(n^{2+\epsilon} b^{1+\epsilon})$.

   \textit{Step 3 (compute GCDs).}
   For each $p \leq Y$, we compute $\gcd(f_p, f_p') \in (\ZZ/p\ZZ)[x]$
   using a quasilinear time GCD algorithm \cite[Cor.~11.9]{MCA_2013}.
   For each prime this requires
   $O(n^{1+\epsilon})$ ring operations in $\ZZ/p\ZZ$,
   and each ring operation costs $O((\lg p)^{1+\epsilon})$ bit operations.
   The hypothesis $n \leq b$ implies that $\lg p = O(\lg(nb)) = O(\lg b)$,
   so the cost of computing the GCD is
   $O(n^{1+\epsilon} b^{\epsilon})$ bit operations.
   The total cost over all $O(nb)$ primes
   is therefore $O(n^{2+\epsilon} b^{1+\epsilon})$.

   Finally, we return the least prime $p$ for which $f_p \neq 0$ and
   $\gcd(f_p, f_p') = 1$.
   As shown above, such a prime exists and satisfies $p \leq X$.
\end{proof}

We now give a deterministic root-finding algorithm for the squarefree case.
\begin{prop}
   \label{prop:squarefree-case}
   Let $b \geq n \geq 1$ be integers, and let $f \in \ZZ[x]$
   be a squarefree polynomial of degree $n$ such that $\supnorm{f} \leq 2^b$.
   Then we may find all integer roots of $f$ in time
   \[
      O(n^{2+\epsilon} b^{1+\epsilon}).
   \]
\end{prop}
\begin{proof}
   As above, let $f_p \in (\ZZ/p\ZZ)[x]$ denote the reduction of $f$ modulo $p$.
   We first invoke Proposition \ref{prop:find-p} to find a prime $p = O(nb)$
   such that $f_p$ is nonzero and squarefree.
   Then we perform the following steps.

   \textit{Step 1 (find roots mod $p$).}
   Compute the roots of $f_p$ in $\ZZ/p\ZZ$ by brute force,
   i.e., by evaluating $f_p(i)$ for $i = 0, \ldots, p-1$.
   Note that the integer roots of $f$ correspond to \emph{distinct}
   roots of $f_p$, thanks to the squarefreeness of $f_p$.
   Each $f_p(i)$ may be evaluated in time $O(n^{1+\epsilon} b^{\epsilon})$,
   so the cost of this step is
   $O(p n^{1+\epsilon} b^{\epsilon}) = O(n^{2+\epsilon} b^{1+\epsilon})$.

   \textit{Step 2 (find roots mod $p^k$).}
   Let $\bar f \in (\ZZ/p^k\ZZ)[x]$ be the reduction of $f$ modulo $p^k$,
   where $k$ is chosen to be the smallest integer such that
   \begin{equation}
      \label{eq:pk}
      p^k > (n+1)^{1/2} \, 2^{n+b+1}.
   \end{equation}
   Applying Proposition \ref{prop:hensel} to $\bar f$,
   we lift each of the roots of $f_p$ found in Step 1 to a root of~$\bar f$.
   The uniqueness claim in Proposition \ref{prop:hensel} implies that
   the resulting set of lifted roots in $\ZZ/p^k\ZZ$
   includes the reductions modulo $p^k$
   of all of the actual integer roots of $f$.
   To estimate the complexity,
   observe that $p^k \leq p (n+1)^{1/2} \, 2^{n+b+1}$, so
   \[
      \lg(p^k) = O(\lg p + \lg(n+1) + n + b) = O(b).
   \]
   The cost of lifting each root is therefore
   $O(n \lg(p^k)^{1+\epsilon}) = O(n b^{1+\epsilon})$,
   and the total cost of this step is $O(n^2 b^{1+\epsilon})$.

   \textit{Step 3 (check roots in $\ZZ$).}
   For each root $\bar r \in \ZZ/p^k \ZZ$ of $\bar f$ found in Step~2,
   we determine whether it arises from a genuine integer root of $f$
   as follows.
   We first lift $\bar r$ to a candidate root $r^* \in \ZZ$ satisfying
   $r^* \equiv \bar r \pmod{p^k}$ and $r^* \in [-\frac12 p^k, \frac12 p^k)$.
   We next divide $\bar f$ by $x - \bar r$ to obtain a polynomial
   $\bar g \in (\ZZ/p^k\ZZ)[x]$ such that
   $\bar f(x) = (x - \bar r) \bar g(x)$,
   and we lift $\bar g$ to a polynomial $g^* \in \ZZ[x]$ satisfying
   $g^* \equiv \bar g \pmod{p^k}$
   and whose coefficients also all lie in $[-\frac12 p^k, \frac12 p^k)$.
   We then multiply $x-r^*$ by $g^*(x)$ (in $\ZZ[x]$)
   and check whether we obtain $f$.
   If so, then $f(r^*) = 0$,
   so $r^*$ must be the integer root corresponding to~$\bar r$.
   Otherwise, as we will see in the next paragraph,
   this $\bar r$ does not correspond to any integer root and we may ignore it.
   This procedure requires $O(n)$ operations on integers of $O(b)$ bits,
   i.e., $O(n b^{1+\epsilon})$ bit operations,
   so the total cost over all roots is $O(n^2 b^{1+\epsilon})$.
   
   We now prove that the procedure described above
   does in fact find all integer roots.
   (The following argument is adapted from \cite[\S15.6]{MCA_2013}.)
   Let $r \in \ZZ$ be a root of~$f$.
   Then $|r| \leq \supnorm{f} \leq 2^b$
   (as $r$ divides the constant term of $f$),
   and $f$ factors as $f(x) = (x-r)g(x)$ for some $g \in \ZZ[x]$
   satisfying $\supnorm{g} \leq (n+1)^{1/2} \, 2^{n+b}$
   (by Lemma \ref{lem:mignotte}).
   In particular, \eqref{eq:pk} ensures that
   $|r| < p^k/2$ and $\supnorm{g} < p^k/2$.
   Let $\bar r \in \ZZ/p^k\ZZ$ be the root of $\bar f$ corresponding to $r$,
   and let $r^* \in \ZZ$ and $g^* \in \ZZ[x]$ be the quantities computed
   in Step~3 for this $\bar r$.
   Then $r^* \equiv \bar r \equiv r \pmod{p^k}$,
   so we must have $r^* = r$,
   since both sides lie in $[-\frac12 p^k, \frac12 p^k)$.
   Similarly, we have
   \[
      g^*(x) \equiv \bar g(x) = \bar f(x)/(x-\bar r)
         \equiv f(x)/(x-r) = g(x) \pmod{p^k},
   \]
   so again we must have $g^* = g$ as the coefficients on both sides lie in
   $[-\frac12 p^k, \frac12 p^k)$.
   Therefore $(x-r^*)g^*(x) = f(x)$,
   and the procedure does indeed recover $r$.
\end{proof}

\begin{rem}
   Loos \cite{Loos-zeroes} imposes the additional requirement that
   $p$ should not divide the leading coefficient of $f$,
   to ensure that $\deg f_p = \deg f$.
   This is because he is searching for \emph{rational} roots,
   not just integral roots.
   Our algorithm may also be easily adapted to this case.
\end{rem}

\begin{rem}
   \label{rem:quasilinear1}
   An interesting question is whether the complexity bound in
   Proposition \ref{prop:squarefree-case} can be improved to quasilinear,
   i.e., to $O(n^{1+\epsilon} b^{1+\epsilon})$ bit operations.
   There are two main obstructions to this.

   First, although $f_p \in (\ZZ/p\ZZ)[x]$ is squarefree
   for almost all primes $p$,
   it is difficult to predict in advance for which $p$ this will occur.
   Consequently, in the proof of Proposition \ref{prop:find-p} we were forced
   to test \emph{every} prime up to $O(nb)$.
   If we allow probabilistic algorithms,
   then we can find a suitable prime with high probability by randomly
   selecting $p$ in the range $2 \leq p \leq X'$ for some $X' = O(nb)$.
   This allows us to find a suitable $p$ in \emph{expected} quasilinear time.
   The complexity of Steps 1 and 2 in Proposition \ref{prop:squarefree-case},
   i.e., finding the roots modulo $p$ and lifting them to $\ZZ/p^k \ZZ$,
   can also be improved to (deterministic) quasilinear time
   by means of fast multipoint evaluation techniques.
   The resulting algorithm is quite similar to the root finding algorithm
   presented in \cite[Thm.~15.21]{MCA_2013}.

   The second obstruction concerns Step 3 of
   Proposition \ref{prop:squarefree-case}, namely,
   checking which of the candidate integer roots are in fact roots of $f$.
   We do not know how to carry out this step rigorously in quasilinear time,
   even allowing randomised algorithms.
   A similar issue occurs in \cite[Thm.~15.21]{MCA_2013},
   where the last term of the given complexity bound corresponds
   in our notation to $O(n^{2+\epsilon} b^{1+\epsilon})$.
   In the discussion following that theorem,
   von zur Gathen and Gerhard suggest testing the candidate roots modulo
   a small prime (different to $p$) as a way to
   quickly rule out incorrect candidates.
   This idea can be turned into a ``Monte Carlo'' algorithm:
   one would randomly choose a small prime $q$,
   compute $f(r^*) \pmod q$ for all candidate roots $r^*$,
   and declare the ones for which $f(r^*) \equiv 0 \pmod q$
   to be the true roots.
   We suspect that in this way one can obtain a
   quasilinear expected running time
   with an exponentially small probability of failure,
   but we have not checked the details.
\end{rem}

\subsection{The general case}

In order to prove Theorem \ref{thm:root-finding},
we must first discuss the computation of GCDs in $\ZZ[x]$.

Let $f, g \in \ZZ[x]$ and let $h \coloneqq \gcd(f,g)$.
The idea of the ``heuristic GCD'' algorithm \cite{CGG-gcdheu}
is to use an \emph{integer} GCD algorithm to compute $\gcd(f(N), g(N))$
for some choice of evaluation point $N \in \ZZ$.
If we are lucky, then $\gcd(f(N), g(N))$ will actually be equal to $h(N)$,
and we may simply read off the coefficients of $h(x)$ from $h(N)$,
provided that $N$ is not too small.
However, it is possible for $\gcd(f(N), g(N))$ to contain extraneous factors
unrelated to $h(x)$.
Usually these extraneous factors are small but in rare circumstances
they can be very large.
The algorithm can be made to tolerate extraneous factors up to a given size
by taking larger values of $N$,
at the expense of running more slowly.
In practice, one usually takes a fairly small value of $N$,
accepting a small chance of failure in order to get a fast algorithm.
In the next result we work at the other extreme,
taking $N$ so large that the algorithm is guaranteed to work in all cases.
We thereby obtain a GCD algorithm that is deterministic
and completely rigorous (although unfortunately quite slow in practice).

\begin{prop}
   \label{prop:gcd}
   Let $b \geq n \geq 1$ be integers.
   Let $f, g \in \ZZ[x]$ be nonzero polynomials such that
   $\deg f, \deg g \leq n$ and $\supnorm{f}, \supnorm{g} \leq 2^b$.
   Assume that at least one of $f$ and $g$ is primitive.
   Define
   \[
      h \coloneqq \gcd(f,g) \in \ZZ[x],
      \qquad \tilde f \coloneqq f / h \in \ZZ[x],
      \qquad \tilde g \coloneqq g / h \in \ZZ[x].
   \]
   Then
   \begin{equation}
      \label{eq:hfg-bound}
      \tsupnorm{h}, \tsupnorm{\tilde f}, \tsupnorm{\tilde g}
         \leq (n+1)^{1/2} \, 2^{n+b},
   \end{equation}
   and given $f$ and $g$ as input,
   we may compute $h$, $\tilde f$ and $\tilde g$ in time
   \[
      O(n^{2+\epsilon} b^{1+\epsilon}).
   \]
\end{prop}
\begin{proof}
   The inequalities \eqref{eq:hfg-bound} follow immediately from
   Lemma \ref{lem:mignotte},
   as $\tilde f$ and $\tilde g$ are divisors of $f$ and $g$ respectively,
   and $h$ is a divisor of both.
   
   For any integer $N$ we have
   $f(N) = \tilde f(N) h(N)$ and $g(N) = \tilde g(N) h(N)$, so
   \[
      \gcd(f(N),g(N)) = \delta(N) \cdot h(N)
      \qquad \text{where } \delta(N) \coloneqq \gcd(\tilde f(N), \tilde g(N)).
   \]
   Writing $h(x) = h_0 + h_1 x + \cdots + h_n x^n$, this becomes
   \begin{equation}
      \label{eq:gcd}
      \gcd(f(N), g(N))
         = \delta(N) h_0 + \delta(N) h_1 N + \cdots + \delta(N) h_n N^n.
   \end{equation}

   We may bound the quantities $\delta(N) h_i$ independently of $N$ as follows.
   (This argument is adapted from \cite[Thm.~4]{DP-heugcd}.)
   Since $\tilde f$ and $\tilde g$ are relatively prime,
   their resultant $R \coloneqq \res(\tilde f, \tilde g)$ is nonzero,
   and there exist polynomials $r, s \in \ZZ[x]$ such that
   \[
      r(x) \tilde f(x) + s(x) \tilde g(x) = R.
   \]
   Substituting $x = N$ shows that $\delta(N) \divides R$.
   Applying Lemma \ref{lem:resultant-bound}, we obtain
   \[
      |\delta(N)| \leq |R|
         \leq (n+1)^n \tsupnorm{\tilde f}^n \tsupnorm{\tilde g}^n
         \leq (n+1)^{2n} \, 2^{2n^2 + 2nb}.
   \]
   Therefore the quantities $\delta(N) h_i$ are bounded by
   \[
      |\delta(N) h_i| \leq |\delta(N)| \supnorm{h}
         \leq (n+1)^{2n+\frac12} \, 2^{2n^2 + 2nb + n + b}.
   \]

   We may now describe the actual algorithm for computing
   $h$, $\tilde f$ and $\tilde g$.

   \textit{Step 1.}
   Set $N \coloneqq 2^c$ where
   \[
      c \coloneqq (2n+1) \lceil \lg(n+1) \rceil + 2n^2 + 2nb + n + b + 2.
   \]
   As shown above,
   $|\delta(N) h_i| \leq 2^{c-2}$ for all $i$.
   Notice also that $c = O(nb)$.

   \textit{Step 2.}
   We compute $f(N)$ and $g(N)$,
   which amounts to concatenating the coefficients of $f$ and $g$
   with appropriate zero-padding
   (or one-padding in the case of negative coefficients).
   The integers $f(N)$ and $g(N)$ have bit size $O(nc) = O(n^2 b)$,
   and the concatenation may be performed in linear time,
   i.e., in time $O(n^2 b)$.

   \textit{Step 3.}
   We compute $\gcd(f(N), g(N))$ using a quasilinear time GCD algorithm
   (see for example \cite{SZ-gcd}).
   This requires time
   $O((n^2 b)^{1+\epsilon}) = O(n^{2+\epsilon} b^{1+\epsilon})$.

   \textit{Step 4.}
   We read off the coefficients $\delta(N) h_i$ from \eqref{eq:gcd}.
   This is possible thanks to the bound $|\delta(N) h_i| \leq 2^{c-2}$,
   i.e., the coefficients do not ``overlap''.
   (In more detail,
   we may first read off $\delta(N) h_0$ from the lowest $c$ bits,
   i.e., by reading \eqref{eq:gcd} modulo $N = 2^c$.
   After subtracting off this term,
   we may read off $\delta(N) h_1$ from the next $c$ bits, and so on.)
   This requires linear time $O(n^2 b)$.

   \textit{Step 5.}
   Since we assumed that at least one of $f$ and $g$ is primitive,
   $h$ is also primitive.
   We may therefore compute $\delta(N)$ by taking the GCDs of the integers
   $\delta(N) h_0, \ldots, \delta(N) h_n$.
   Each pairwise GCD requires time $O(c^{1+\epsilon})$,
   so the total time required for this step is
   $O(n c^{1+\epsilon}) = O(n^{2+\epsilon} b^{1+\epsilon})$.

   \textit{Step 6.}
   We now recover $h(N) = \gcd(f(N),g(N)) / \delta(N)$,
   and then $\tilde f(N) = f(N) / h(N)$ and $\tilde g(N) = g(N) / h(N)$.
   Using a quasilinear time integer division algorithm,
   this requires time $O(n^{2+\epsilon} b^{1+\epsilon})$.
   Finally, we read off the coefficients of $\tilde f$ and $\tilde g$
   from $\tilde f(N)$ and $\tilde g(N)$, in a similar manner to Step 4.
\end{proof}

\begin{rem}
   \label{rem:quasilinear2}
   To the best of the authors' knowledge,
   it is not known how to improve the complexity bound in
   Proposition \ref{prop:gcd} to quasilinear without giving up on determinism.
   Several randomised quasilinear-time algorithms are known.
   Sch\"onhage \cite{Sch-prob-gcds} analyses a variant of the
   heuristic GCD algorithm in which the evaluation point is chosen randomly.
   Another approach is to compute the GCD modulo a collection of
   randomly chosen small primes \cite[Alg.~6.38]{MCA_2013}.
\end{rem}

We may now prove our main root-finding result.
\begin{proof}[Proof of Theorem \ref{thm:root-finding}]
   We are given as input $f \in \ZZ[x]$, not necessarily squarefree,
   with $\deg f = n$ and $\supnorm{f} \leq 2^b$, where $b \geq n \geq 1$.
   
   We first compute the GCD of the coefficients of $f$,
   and remove this common factor.
   Clearly this can be done in time $O(n^{1+\epsilon} b^{1+\epsilon})$,
   and we may subsequently assume that $f$ is primitive.
   
   Let $g \coloneqq f'$.
   Then $\deg g \leq n$ and
   $\supnorm{g} \leq n \supnorm{f} \leq 2^{b'}$
   where $b' \coloneqq b + \lceil \lg n \rceil = O(b)$.
   Applying Proposition \ref{prop:gcd},
   we may compute $\tilde f = f/\gcd(f,f') \in \ZZ[x]$ in time
   $O(n^{2+\epsilon} (b')^{1+\epsilon}) = O(n^{2+\epsilon} b^{1+\epsilon})$.
   Then $\tilde f$ is squarefree and has the same integer roots as $f$.
   Moreover we have $\deg \tilde f \leq n$ and
   $\tsupnorm{\tilde f} \leq (n+1)^{1/2} \, 2^{n+b'} \leq 2^{b''}$
   where $b'' \coloneqq n + b' + \lceil \lg(n+1) \rceil = O(b)$.
   
   Finally, we apply Proposition \ref{prop:squarefree-case} to $\tilde f$.
   The running time is
   $O(n^{2+\epsilon} (b'')^{1+\epsilon}) = O(n^{2+\epsilon} b^{1+\epsilon})$.
\end{proof}

\bibliographystyle{amsalpha}
\bibliography{rpowerdivisors}

\end{document}